\documentclass{amsart}

\usepackage{amscd,amsmath,amssymb,amsfonts,bbm, calligra, xspace}
\usepackage[T1]{fontenc}
\usepackage{lmodern}
\usepackage{mathtools}
\usepackage{enumitem}
\usepackage{multicol}
\usepackage[all]{xy}
\usepackage{mathrsfs}
\usepackage{footmisc}
\usepackage{tikz}
\usetikzlibrary{cd,arrows,positioning}

\usepackage{xcolor}
\definecolor{darkblue}{rgb}{0,0,0.8}
\definecolor{darkgreen}{rgb}{0,0.4,0}
\usepackage[colorlinks=true,linkcolor=darkblue, citecolor=darkgreen,  urlcolor=darkgreen, unicode,pdfborder={0 0 0},final]{hyperref}

\newtheorem{thm}{Theorem}[section]
\newtheorem{prop}[thm]{Proposition}

\newtheorem{lem}[thm]{Lemma}
\newtheorem{cor}[thm]{Corollary}

\theoremstyle{definition}
\newtheorem{Def}[thm]{Definition}

\theoremstyle{remark}
\newtheorem{rem}[thm]{Remark}

\newtheorem{ex}[thm]{Example}

\numberwithin{equation}{section}

\newcommand{\NN}{\mathrm{N}}
\newcommand{\NE}{\mathrm{NE}}
\newcommand{\Br}{\mathrm{Br}}
\newcommand{\Spec}{\mathrm{Spec}}
\newcommand{\Gal}{\mathrm{Gal}}
\newcommand{\Jac}{\mathrm{Jac}}
\newcommand{\CH}{\mathrm{CH}}
\newcommand{\nr}{\mathrm{nr}}

\def\A{\mathbb A}

\def\Z{\mathbb Z}
\def\C{\mathbb C}

\def\Q{\mathbb Q}
\def\P{\mathbb P}
\def\R{\mathbb R}

\def\bS{\mathbb S}

\def\cL{\mathcal{L}}
\def\cO{\mathcal{O}}
\def\cE{\mathcal{E}}

\def\cH{\mathcal{H}}

\def\cU{\mathcal{U}}
\def\cV{\mathcal{V}}

\def\cX{\mathcal{X}}

\def\ok{\overline{k}}

\def\wF{\widetilde{F}}

\def\whS{\hat{S}}
\def\whX{\hat{X}}

\newcommand{\isoto}{\myxrightarrow{\,\sim\,}}
\makeatletter
\def\myrightarrow{{\setbox\z@\hbox{$\rightarrow$}\dimen0\ht\z@\multiply\dimen0 6\divide\dimen0 10\ht\z@\dimen0\box\z@}}
\def\myrightarrowfill@{\arrowfill@\relbar\relbar\myrightarrow}
\newcommand{\myxrightarrow}[2][]{\ext@arrow 0359\myrightarrowfill@{#1}{#2}}
\makeatother

\makeatletter
\def\@tempa#1{\@xp\@tempb\meaning#1\@nil#1}
\def\@tempb#1>#2#3 #4\@nil#5{%
  \@xp\ifx\csname#3\endcsname\mathaccent
    \@tempc#4?"7777\@nil#5%
  \else
    \PackageWarningNoLine{amsmath}{%
      Unable to redefine math accent \string#5}%
  \fi
}
\def\@tempc#1"#2#3#4#5#6\@nil#7{%
  \chardef\@tempd="#3\relax\set@mathaccent\@tempd{#7}{#2}{#4#5}}

\@tempa\hat
\makeatother

\begin{document}

\title[On the rationality of some real threefolds]{On the rationality of some real threefolds}

\author{Olivier Benoist}
\address{D\'epartement de math\'ematiques et applications, \'Ecole normale sup\'erieure, CNRS,
45 rue d'Ulm, 75230 Paris Cedex 05, France}
\email{olivier.benoist@ens.fr}

\author{Alena Pirutka}
\address{Courant Institute of Mathematical Sciences, New York University, New York, 10012, U.S.A.}
\email{pirutka@cims.nyu.edu}

\renewcommand{\abstractname}{Abstract}

\begin{abstract}
We study the rationality of some geometrically rational three-dimensional conic and quadric surface bundles, 
 defined over the reals and more general real closed fields, for which the real locus is connected and the intermediate Jacobian obstructions to rationality vanish.  We obtain both negative and positive results, using unramified cohomology and birational rigidity techniques, as well as concrete rationality constructions.
\end{abstract}

\maketitle

\section*{Introduction}

An algebraic variety $X$ of dimension $n$ over a field $k$ is said to be $k$\textit{-rational} if it is birational to the affine space $\A^n_k$ (equivalently,  if its function field $k(X)$ is $k$\nobreakdash-isomorphic to~$k(t_1,\dots,t_n)$).  Deciding whether a given variety $X$ is $k$-rational is an  important and often subtle problem. When $k$ is algebraically closed (say,~${k=\C}$), full answers have been known for a long time
if $n\leq 2$ (for surfaces,  this is Castel\-nuovo's theorem \cite{Castelnuovo}, see~\cite[Theorem V.1]{Beauville}).
The first nontrivial examples of nonrational varieties of dimension $3$ over $\C$ appeared in the 70's:
\begin{enumerate}[label=(\roman*)]
\item 
\label{i}
smooth cubic threefolds,  by Clemens--Griffiths \cite{CG};
\item
\label{ii}
 smooth quartic threefolds,  by Iskovskikh--Manin \cite{IM};
\item
\label{iii}
 some quartic double solids, by Artin--Mumford \cite{AM}.
\end{enumerate}
The techniques used to detect nonrationality vary: \ref{i} relies on intermediate Jacobians, \ref{ii} on birational automorphism groups and \ref{iii} on the Brauer group.

When $k$ is not algebraically closed, with algebraic closure $\ok$, the rationality problem is particularly interesting for varieties that are $\ok$-rational: it then has an arithmetic flavor. Again, full answers are known if $n\leq 2$ (for surfaces, by the works of Segre,  Manin and Iskovskikh \cite{Segre, Maninperfect, Iskovskikhimperfect},  
relying on the birational rigidity techniques also used in \ref{ii},  see \cite[Proposition 4.16]{BW2} for a precise statement).  Thanks to the development over nonclosed fields of obstructions to rationality based on intermediate Jacobians \cite{BW1, HT1, HT2, BW2, KP}, in the spirit of \ref{i}, the case of threefolds has recently seen a flurry of activity.

Our aim is to investigate the $k$-rationality of some $\ok$-rational threefolds that are at the edge of the available techniques,  and in particular out of reach of the intermediate Jacobian obstructions to rationality.  We focus on the simplest nonclosed fields: the field $\R$ of real numbers (the rationality problem for real threefolds features prominently in \cite{BW1, HT1,JJ, FJSVV,CTP, FJSVV2}), and more generally real closed fields (those fields whose absolute Galois group is $\Z/2$; they admit a field ordering whose nonnegative elements are exactly the squares \mbox{\cite[\S 1.2]{BCR}}).

\vspace{1em}

We consider two concrete sets of equations over a real closed field $R$. The first were introduced by Colliot-Th\'el\`ene and Pirutka in \cite{CTP} and read
\begin{equation}
\label{eq1}
x^2+y^2+z^2=u\cdot p(u),
\end{equation}
where $p\in R[u]$ is a nonnegative separable polynomial of even degree $d\geq 2$.  As is explained in the introduction of \cite{CTP}, the varieties (\ref{eq1}) have been chosen among all quadric surface bundles over $\P^1_R$ so that none of the available techniques sheds any light on their rationality.
As a matter of fact, over the reals, none of them are known to be rational (or even stably rational), and none of them are known not to be rational.  It is however shown in \cite{CTP} that many of them are universally $\CH_0$-trivial; see \cite[Th\'eor\`emes 5.1, 5.3 and 7.2]{CTP}.

The second set of equations are conic bundles over $\P^2_R$ originating from the article~\cite{BW1} of Benoist and Wittenberg. They have the form
\begin{equation}
\label{eq2}
x^2+y^2=f(v,w),
\end{equation}
where $f\in R[v,w]$ is a polynomial of even 
degree $d\geq 2$ such that the closure of~$\{f=0\}$ in $\P^2_R$ is a nodal rational curve with normalization $\Delta$, and either
\begin{enumerate}[label=(\ref{eq2}\alph*)]
\item
\label{eqa} 
$\Delta(R)\neq\varnothing$ (so $\Delta\simeq\P^1_R$); or
\item
\label{eqb}
$\Delta(R)=\varnothing$ (so $\Delta$ is an anisotropic conic) and $f$ is nonnegative.
\end{enumerate}

  The varieties defined by~(\ref{eq1}), \ref{eqa} and \ref{eqb} admit smooth projective models that are conic bundles over an $R$-rational surface $S$, with smooth connected rami\-fication locus $\Delta\subset S$, and residue $-1$ along~$\Delta$ (in case (\ref{eq1}), the curve $\Delta$ is the hyperelliptic curve $\{z^2=u\cdot p(u)\}$; in case~(\ref{eq2}), it is a rational curve).  
In particular, they are geometrically rational (as the residue $-1$ becomes a square over $R[\sqrt{-1}]$).  By the nonnegativity hypotheses in~(\ref{eq1})~and~\ref{eqb}, these smooth projective models also have nonempty and semi-algebraically connected real loci (connected if $R=\mathbb R$, see  \cite[Definition 2.4.2]{BCR} for a general real closed field), which is a nec\-es\-sary condition of rationality.   

The rationality of conic bundles of this form was investigated in~\cite[\S 3.3]{BW1} using intermediate Jacobians (by analogy with Beauville's classical work \cite{BeauvillePrym} 
over algebraically closed fields).  In particular, they are not $R$-rational if $\Delta$ is nonhyperelliptic.  
In contrast, equations (\ref{eq1}) and (\ref{eq2}) have been chosen so intermediate Jacobians cannot be used to disprove rationality, as we now briefly explain. In case~(\ref{eq2}), the intermediate Jacobian of the aforementioned model is trivial.
In case~(\ref{eq1}),  the arguments of \cite[Propositions 3.2 and 3.4 (iii)]{BW1} show that the intermediate Jacobian of this model is isomorphic to the Jacobian~$\Jac(\Gamma)$ of the hyperelliptic curve~$\Gamma:=\{z^2+u\cdot p(u)=0\}$. As $\Gamma(R)$ is nonempty and semi-algebraically connected,  any torsor under $\Jac(\Gamma)$ is trivial (see \cite[line below~(4.4)]{Knebusch}). 
All the known obstructions to rationality related to intermediate Jacobians therefore vanish.

\vspace{1em}

Our first result is that the varieties (\ref{eq1}) are not always rational.

\begin{thm}[Theorem \ref{th1}]
\label{th1intro}
For each even $d\geq 2$, there exists a variety $X$ with equation (\ref{eq1}) over 
the real closed field $R:=\cup_{n\geq 1}\R((t^{\frac{1}{n}}))$ that is not $R$-rational.
\end{thm}

Our examples are given by concrete equations. They are not even stably rational.

The proof uses unramified cohomology groups. These higher degree analogues of 
the Brauer group give rise to generalizations of the Artin--Mum\-ford approach~\ref{iii} to irrationality,  as was discovered by Colliot-Th\'el\`ene and Ojanguren in \cite{CTO}. To be precise, we rely on an unramified cohomology class $\alpha\in H^3_{\nr}(R(X\times Y)/R,\Z/2)$ (where $Y$ is some algebraic curve over $R$) whose vanishing was shown in \cite{CTP}  to be equivalent to the universal $\CH_0$-triviality of $X$, and hence whose nontriviality is an obstruction to the rationality of $X$ (see \cite[Th\'eor\`eme 4.8]{CTP} for $k=\mathbb R$ or \cite[Th\'eor\`eme 3.5]{CTP} with $a=b=-1$ for a general real closed field).  We show that $\alpha$ is nonzero as follows.  The varieties $X$ and $Y$ and the class~$\alpha$ are all defined over~$\R((t))$.  
It suffices to show that $\alpha$ is nonzero over $\R((t^{\frac{1}{n}}))$ for all $n\geq 1$. To do so, we show that the class $\alpha$ is ramified along some divisor supported on the special fiber of a model of~$X\times Y$ over $\R[[t^{\frac{1}{n}}]]$,  using \textit{different} divisors for different values of $n$.
We believe that this way of detecting the nontriviality of an unramified cohomology class is novel.

We leave entirely open the question of the rationality of the varieties~(\ref{eq1}) over~$\R$. 
Nevertheless, Theorem~\ref{th1intro} yields some information on this rationality problem. 
It implies that there cannot exist a single rationality construction (or even a finite number of them) showing that all the varieties $X$ defined by (\ref{eq1}) with $p\in\R[u]$ of fixed even degree~$d\geq 2$ are $\R$-rational.
In particular, we obtain the next corollary.

\begin{cor}[Corollary \ref{correal}]
\label{correalintro}
Fix $d\geq 2$ even and $\delta\geq 0$.  There exists a variety~$X$ over $\R$ with equation (\ref{eq1}) for this value of $d$ such that there is no birational map~$X\dashrightarrow\A^3_{\R}$ given by rational functions of degrees $\leq\delta$ in the variables~$(x,y,z,u)$.
\end{cor}

\vspace{.5em}

Our second result demonstrates that the equations \ref{eqa} and \ref{eqb} 
are rational in low degree.

\begin{thm}[Theorem \ref{th2}]
\label{th2intro}
The algebraic varieties defined by (\ref{eq2}) with $d\leq4$
 are always $R$-rational, over any real closed field $R$.
\end{thm}

When $d=2$, the statement of Theorem \ref{th2intro} is obvious, as the varieties under consideration are irreducible quadrics with a smooth rational point.  When $d=4$, our proof proceeds by giving direct rationality constructions for these varieties,  based on a detailed study of nodal rational quartic curves over $R$.

\vspace{1em}

We finally show that the varieties (\ref{eq2}) are not rational in high degree.

\begin{thm}[Theorem \ref{th3}]
\label{th3intro}
The algebraic varieties with equation (\ref{eq2}) are never $R$-rational if $d\geq 12$, over any real closed field $R$.
\end{thm}

As we will explain below,  it is reasonable to conjecture that Theorem \ref{th3intro} extends to all degrees $d\geq 6$,  which would fit nicely with Theorem \ref{th2intro}.

Theorem \ref{th3intro} is proved using birational rigidity techniques, and more precisely by applying a variant over a nonclosed field (Theorem \ref{th4intro} below) of a theorem of Sarkisov \cite[Main Theorem]{Sarkisov}.  
Before stating it,  we give a few definitions.  \textit{Mori fiber spaces}~${\pi:X\to S}$ are a certain kind of Fano fibrations. To be precise, they are morphisms with geometrically connected positive-dimensional fibers between connected normal projective varieties over $k$,  that have relative Picard rank~$1$,  and such that $X$ is $\Q$-factorial and terminal with $\pi$-ample anticanonical divisor (see Defi\-ni\-tion~\ref{defMFS}).  Important examples for us are \textit{standard conic bundles}: those flat Mori fiber spaces whose generic fiber is a conic, with $X$ and $S$ smooth and snc degeneration divisor (see Defi\-ni\-tion~\ref{defstandard}). 
A Mori fiber space~$\pi:X\to S$ is said to be \textit{birationally superrigid} if any birational map~$\phi:X\dashrightarrow X'$ to the total space of another Mori fiber space~$\pi':X'\to S'$ induces an isomorphism between the generic fibers of $\pi$ and $\pi'$ (see Definition \ref{defbirrig}).

\begin{thm}[Theorem \ref{th4}]
\label{th4intro}
Let $k$ be a field of characteristic $0$. Let $\pi:X\to S$ be a standard conic bundle of dimension $3$ over $k$ with discriminant curve $\Delta\subset S$. 
Assume that $|4K_S+\Delta|\neq\varnothing$. Then $\pi:X\to S$ is birationally superrigid.
\end{thm}

To deduce Theorem \ref{th3intro},  we consider a birational model of~(\ref{eq2}) that is the total space of a standard conic bundle $\pi:X\to S$ (whose base $S$ is the blow-up of $\P^2_R$ at the nodes of the projective closure of $\{f=0\}$). The hypothesis that $d\geq 12$ allows us to verify the hypothesis that $|4K_S+\Delta|\neq\varnothing$ in Theorem \ref{th4intro}.  We therefore deduce that $\pi:X\to S$ is birationally superrigid,  and hence that $X$ is not rational.

When $k$ is algebraically closed, Theorem \ref{th4intro} is due to Sarkisov \cite{Sarkisov} (see also the articles of Corti \cite{Corti1,Corti2} and Prokhorov \cite{Prokhorov}).  It is clear to the experts that the proof of Sarkisov's theorem given in~\cite{Corti1, Corti2,Prokhorov} goes through over a nonclosed field $k$ (see e.g.\ \cite[Theorem 27]{KollarCY}).  Since this fact is crucial to our application to Theorem \ref{th3intro}, and since we could not find detailed explanations in the literature, we devote Sections~\ref{sec3} and \ref{sec4} to spelling out the minor changes to make to the proofs of \cite{Corti1, Corti2,Prokhorov} work over nonclosed fields.

We insist that Theorem \ref{th4intro} is not a formal consequence of Sarkisov's theorem over~$\ok$, because being a Mori fiber space is not invariant under extensions of the base field (neither the relative Picard rank $1$ condition nor the $\Q$-factoriality hypothesis are preserved in general; in particular, for standard models of varieties in Theorem \ref{th3intro}  the relative Picard rank is $1$ over $R$, but greater than $1$ over the algebraic closure, see Lemma \ref{lemstandardmodel} and its proof). This also explains why we are able to detect the nonrationality of some $\ok$-rational varieties. 

To the best of our knowledge, Theorem \ref{th3intro} is the first application of birational rigidity techniques to $k$-rationality problems for $\ok$-rational varieties in dimen\-sion~${\geq 3}$.  However, these techniques have been applied in this precise way in dimension $2$ for a long time.  Indeed, they have been at the heart of the birational classification of geometrically rational surfaces over a nonclosed field $k$, ever since Segre's proof that cubic surfaces of Picard rank $1$ over $k$ are never~$k$-rational \cite[Theorems 3 and 5]{Segre}. These techniques have also been applied in the equivariant setting, in particular in dimen\-sion~${\geq 3}$: for recent progress on equivariant birational rigidity and automorphism groups of rings of invariants see, e.g., \cite{CCG1, CCG2, AG1, BCDP, PG1, KG1}.

Sarkisov's theorem is not expected to be optimal.  Corti \cite[Problem 4.13]{Corti2} asks whether one
could replace the condition that $|4K_S+\Delta|\neq\varnothing$ by the weaker condition that $|3K_S+\Delta|\neq\varnothing$.  In addition, it is conjectured by Shokurov \cite[Conjecture 10.3]{Shokurov} (see also \cite[Conjecture 1.2]{Prokhorov}) that under the even weaker hypothesis that $|2K_S+\Delta|\neq\varnothing$, the variety $X$ should not be rational.  If the analogues of these questions over nonclosed fields had positive answers, we could deduce that the varieties with equation (\ref{eq2}) are never rational for $d\geq 6$.

We finally note that birational rigidity techniques are not applicable to the varieties defined by the equations (\ref{eq1}), since they are birational both to Mori fiber spaces that are conic bundles,  and quadric surface bundles.

\subsection*{Acknowledgements}

The first author thanks NYU for its hospitality during his stay in New York in September and October 2024. The second author was partially supported by  NSF grant DMS-2201195. We would like to thank J.-L. Colliot-Th\'el\`ene for several discussions, for  pointing out the  application to singular cubics in Remark \ref{remexample}(ii) and for simplifying the proof of Theorem \ref{th2}(a). We are grateful to J. Koll\'ar and Y. Prokhorov for comments on the manuscript.

\section{Degree $3$ unramified cohomology}
\label{sec1}

In this section, we prove Theorem \ref{th1intro}.

\subsection{A nonrational threefold over a nonarchimedean real closed field}

Recall that a smooth projective variety~$X$ over a field~$k$ is \textit{universally} $\CH_0$\textit{-trivial} if the degree map $\deg:\CH_0(X_l)\to \Z$ is an isomorphism for all field extensions $l$ of~$k$, and
that $X$
 is \textit{stably} $k$\textit{-rational} if $X\times_k\P^N_k$ is $k$-rational for some~$N\geq 0$. Rational varieties are stably rational 
and stably rational varieties are universally $\CH_0$\nobreakdash-trivial (see \cite[Lemma 1.5]{CTPspe}).

\begin{thm}
\label{th1}
Consider the real closed field $R:=\cup_{n\geq 1}\R((t^{\frac{1}{n}}))$.  Fix $e\geq 0$. The smooth projective models of the algebraic variety over $R$ defined by the equation
\begin{equation}
\label{eqR}
x^2+y^2+z^2=(u-1)(u^2+t)\prod_{j=1}^e(ju^2+1)
\end{equation}
are not universally $\CH_0$-trivial, and consequently not stably $R$-rational.
\end{thm}

\begin{rem}\label{remexample}
(i)
The change of variable $u\mapsto u+1$ shows that (\ref{eqR}) is of the form~(\ref{eq1}) with $d=2e+2$. The equation (\ref{eqR}) will be more adapted to our computations.

(ii)
Our simplest example, for $e=0$, has equation $$x^2+y^2+z^2=(u-1)(u^2+t).$$
This example is also  birational to a singular cubic threefold $X$ over $R$, defined by the homogeneous equation
$$x^2v+y^2v+z^2v-(u-v)(u^2+tv^2)=0.$$
The singular locus of $X$ is the anisotropic conic ${x^2+y^2+z^2=u=v=0}$. In particular, the variety $X$ provides a negative answer to  
\cite[Section 7, Remark~7.2]{CTZ} over the real closed field $R$. 
\end{rem}

For $m\geq 1$, let~$K_m$ be the function field of the~$\R[[s]]$-subscheme of $\A^6_{\R[[s]]}$ (with coordinates $(u,v,w,x,y,z)$) defined by the equations
\begin{equation}
\label{equ1}
\left\{ \begin{aligned} 
x^2+y^2+z^2 &= (u-1)(u^2+s^{4m})\prod_{j=1}^e(ju^2+1)\\
w^2 &= (v+1)(v^2+s^{4m})\prod_{j=1}^e(jv^2+1).
\end{aligned} \right.
\end{equation}

The next proposition is key to the proof of Theorem \ref{th1}.

\begin{prop}
\label{propnr}
For all $m\geq 1$, the class $(u+v,-1,-1)\in H^3(K_m,\Z/2)$ does not belong to $H^3_{\nr}(K_m/\R,\Z/2)$.
\end{prop}

\begin{proof}
Consider the $\R[[s]]$-scheme $Z$ defined by the equations 
\begin{equation}
\label{equ2}
\left\{ \begin{aligned} 
x'^2+y'^2+z'^2 &= (s^{2m}u'-1)(u'^2+1)\prod_{j=1}^e(js^{4m}u'^2+1)\\
w'^2 &= (s^{2m-1}v'+1)(v'^2+s^{2})\prod_{j=1}^e(js^{4m-2}v'^2+1),
\end{aligned} \right.
\end{equation}
The change of variables $x=s^{2m}x'$,  $y=s^{2m}y'$, $z=s^{2m}z'$, $u=s^{2m}u'$, $w=s^{2m-1}w'$ and $v=s^{2m-1}v'$ induces a birational isomorphism between (\ref{equ1}) and (\ref{equ2}).  In these new coordinates, the class $(u+v,-1,-1)$ is equal to $(s^{2m-1}(su'+v'),-1,-1)$.

The Cartier divisor $D:=\{s=0\}\subset Z$ is the real algebraic variety with equations
\begin{equation}
\label{eq3}
\left\{ \begin{aligned} 
x'^2+y'^2+z'^2+u'^2+1&=0 \\
w'^2-v'^2&=0.
\end{aligned} \right.
\end{equation}
It has two irreducible components that we denote by $D_1$ (on which $w'=v'$) and~$D_2$ (on which~${w'=-v'}$).  As $D$ is smooth at the generic point of $D_1$, the scheme~$Z$ is regular at the generic point of $D_1$.  Let $\omega$ be the discrete valuation of the function field of $Z$ associated with $D_1$, and let $\partial_{\omega}$ be the corresponding residue map.

Since $\omega(s^{2m-1}(su'+v'))=2m-1$ is odd,  one sees that 
\begin{equation}
\label{residue}
\partial_{\omega}((s^{2m-1}(su'+v'),-1,-1))=(-1,-1)\textrm{ in }H^2(\R(D_1),\Z/2).
\end{equation}
The class (\ref{residue}) vanishes if and only if $-1$ is a sum of $2$ squares in $\R(D_1)$ (see \cite[Propositions 1.3.2 and 4.7.1]{GS}).  As $D_1$ is stably birational to the real anisotropic quadric of dimension ~$3$, whose function field has level $4$ (e.g.\,by Pfister's work \cite[Satz 5]{Pfister}), one cannot write $-1$ as a sum of two squares in $\R(D_1)$. It follows that the residue (\ref{residue}) is nonzero. 
\end{proof}

We may now prove Theorem \ref{th1}.

\begin{proof}[Proof of Theorem \ref{th1}]
Let $W$ be the algebraic variety over $\R((t))$ defined by the system of equations
\begin{equation}
\label{equationsproduct}
\left\{ \begin{aligned} 
x^2+y^2+z^2 &= (u-1)(u^2+t)\prod_{j=1}^e(ju^2+1)\\
w^2 &= (v+1)(v^2+t)\prod_{j=1}^e(jv^2+1).
\end{aligned} \right.
\end{equation}
For all $m\geq 1$, applying Proposition \ref{propnr} with $s=t^{\frac{1}{4m}}$ shows that the cohomology class $(u+v,-1,-1)\in H^3(\R((t^{\frac{1}{4m}}))(W),\Z/2)$ is not unramified over~$\R$, and hence is nonzero.  It follows that the class $(u+v,-1,-1)\in H^3(R(W),\Z/2)$ is also nonzero.

By \cite[Th\'eor\`emes 3.3 et 3.4]{CTP} applied with $k=R$ and $a=b=-1$ (after a change of coordinates $u\mapsto u+1$ and $v\mapsto v-1$),  the class $(u+v,-1,-1)$ belongs to $H^3_{\nr}(R(W)/R,\Z/2)$ and its nonvanishing is equivalent to the smooth projective models of the variety defined by (\ref{eqR}) not being universally $\CH_0$-trivial. Since we have shown that this class is nonzero, the proof is complete.
\end{proof}

\begin{rem}
The above proof of Theorem \ref{th1} relies on \cite[Th\'eo\-r\`eme~3.3]{CTP}, and through it on deep results of Colliot-Th\'el\`ene and Skorobogatov about Chow groups of zero-cycles on quadric bundles \cite{CTSk}. We now explain how to conclude in a more elementary way once we know that the cohomology  class~${\alpha:=(u+v,-1,-1)}$ is unramified over $R$ (by \cite[Th\'eo\-r\`eme~3.3]{CTP}) and nonzero (by Proposition \ref{propnr}).

Let $X$ and $Y$ be the varieties over $R$ respectively defined by the first and the second equation of (\ref{equationsproduct}), so $W_R=X\times_R Y$.  
Assume for contradiction that the variety $X$ is $R$\nobreakdash-rational. 
Set $F:=R(Y)$. Since $\alpha\in H^3_{\nr}(R(W)/R,\Z/2)$, one has a fortiori $\alpha\in H^3_{\nr}(F(X)/F,\Z/2)$. It therefore follows from \cite[Proposition 1.2]{CTO} that $\alpha$ is obtained by restriction from a class $\beta\in H^3(F,\Z/2)$.

Let $o\in X(R)$ be a point with coordinates $(x_0,y_0,z_0,u_0)$ such that $u_0\geq 1$. One has~$\beta=\alpha|_{o}=(u_0+v,-1,-1)$ in $H^3(F,\Z/2)$ (where $o$ is now viewed as an $F$-point of $X$).  The equation of $Y$ shows that $v+1\in R(Y)$ is nonnegative, and hence a sum of two squares in $R(Y)$ by \cite[Theorem 4.1]{Knebusch}. As $u_0-1\geq 0$ is a square in $R$, the element $u_0+v\in R(Y)$ is a sum of three squares in $R(Y)$. In particular, the Pfister form $\langle\langle u_0+v,-1,-1\rangle\rangle$ is isotropic. It therefore follows from a result of Arason \cite[Satz 1.6]{Arason} that $\beta=(u_0+v,-1,-1)=0$, so that $\alpha=0$. This is the required contradiction.
\end{rem}

\subsection{Consequences over the reals}
\label{parimplicationR}

Fix an even integer $d\geq 2$.  Let $B$ be the algebraic variety over $\R$ parametrizing all the separable polynomials $p\in\R[u]$ of degree $d$. Let $\cX\subset \A^4_{\R}\times_{\R}B$ be the real algebraic variety with equation
$$x^2+y^2+z^2=u\cdot p(u),$$
where $(x,y,z,u)$ are the coordinates of $\A^4_{\R}$.
The first projection $\pi:\cX\to B$ is the universal family of algebraic varieties with an equation of the form (\ref{eq1}).

\begin{prop}
\label{propreal}
The following data do not exist:
\begin{enumerate}[label=(\roman*)]
\item a morphism $f:B'\to B$ of algebraic varieties over $\R$ such that $f(B'(\R))$  contains all nonnegative polynomials $p\in B(\R)$;
\item an open subset $\cU\subset\cX':=\cX\times_BB'$ intersecting all the fibers of the base change $\pi':\cX'\to B'$ of $\pi$ by $f$;
\item an open subset $\cV\subset \A^3_{\R}\times_{\R}B'$ and an isomorphism $\cU\isoto\cV$ of varieties over $B'$.
\end{enumerate}
\end{prop}

\begin{proof}
Assume that such data exist.  Let $R$ be the real closed field $\cup_{n\geq 1}\R((t^{\frac{1}{n}}))$. By Theorem \ref{th1},  there exists a nonnegative separable polynomial $p_0\in R[u]$ of degree $d$ such that the variety over $R$ with equation $x^2+y^2+z^2=u\cdot p_0(u)$ is not $R$-rational.
It follows from (ii) and (iii) that $p_0\in B(R)$ cannot belong to $f(B'(R))$.  We deduce from the Tarski--Seidenberg transfer principle (see e.g. \cite[Proposition 5.2.3]{BCR}) that there exists a nonnegative $p\in B(\R)$ that does not belong to $f(B'(\R))$. This contradicts (i).
\end{proof}

Proposition \ref{propreal} means that a single rationality construction (or finitely many of them) would not suffice to show the $\R$-rationality of all algebraic varieties~$X$ over~$\R$ with equation (\ref{eq1}) for a fixed $d$.  In particular, it implies the following.

\begin{cor}
\label{correal}
Fix $d\geq 2$ even and $\delta\geq 0$.  There exists a variety $X$ over $\R$ with equation (\ref{eq1}) for this value of $d$, such that there is no birational map~$X\dashrightarrow\A^3_{\R}$ given by rational functions of degrees $\leq\delta$ in the variables~$(x,y,z,u)$.
\end{cor}

\begin{proof}
Assume for contradiction that no such $X$ exists.
Let $V$ be the algebraic variety over $\R$ parameterizing $7$-tuples $(a_1,b_1,a_2,b_2,a_3,b_3,p)$, where the $a_i$ and the~$b_i$ are elements of $\R[x,y,z,u]$ of degree $\leq\delta$, and $p\in\R[u]$ is separable of degree $d$. The subset of~$V$ parameterizing those $(a_1,b_1,a_2,b_2,a_3,b_3,p)$ such that~$(\frac{a_1}{b_1},\frac{a_2}{b_2}, \frac{a_3}{b_3})$ induces a well-defined birational map between $\{x^2+y^2+z^2=u\cdot p(u)\}$ and the affine space of dimension $3$ is constructible. It can thus be written as the union of finitely many locally closed algebraic subvarieties $(V_i)_{i\in I}$ of $V$. Let~$B'$ be the disjoint union of the $(V_i)_{i\in I}$. The morphism $f:B'\to B$ given by~${f(a_1,b_1,a_2,b_2,a_3,b_3,p)=p}$ is such that $f(B'(\R))$ contains all nonnegative polynomials $p\in B(\R)$, by our assumption. Condition (i) of Proposition \ref{propreal} therefore holds.

In addition, choosing $\cU$ and $\cV$ to be open subsets on which the birational map $(\frac{a_1}{b_1},\frac{a_2}{b_2}, \frac{a_3}{b_3})$ induces an isomorphism, one verifies that conditions (ii) and (iii) of Proposition \ref{propreal} also hold. Proposition \ref{propreal} now gives the required contradiction.
\end{proof}

\section{Rationality of some degree $4$ conic bundles}
\label{sec2}

In this section, we prove Theorem \ref{th2intro}.
Let $R$ be a real closed field and let $$C:=\{f(v,w,z)=0\}$$  be a geometrically irreducible quartic nodal curve in $\mathbb P^2_R$ whose normalization~$\Delta$ is of genus zero. We will write  $f(v,w)=0$ (resp. $f(v,z)=0$, $f(w,z)=0$) for the equation of $C$ in the affine chart $z=1$ (resp. $w=1$, $v=1$). We first have an elementary lemma:

\begin{lem}\label{elem1}
The curve $C$ has three nodes over the algebraic closure of $R$, at least one of the nodes is defined over $R$. Up to a linear change of variables, one may assume that the point $[0:0:1]$ is a node of $C$, and that the quartic $f(v,w,z)$ is given by the equation
$$f(v,w,z) = q(v,w)z^2+c(v,w)z+g(v,w),$$ 
where the homogeneous forms $q,c,g\in R[v,w]$ are of degrees $2,3,4$ respectively.
In addition,
\begin{enumerate}[label=\textrm{(\alph*)}]
\item If all three nodes are defined over $R$, up to a linear change of variables, one may assume that the points $[1:0:0]$, $[0:1:0]$, and $[0:0:1]$ are the nodes of $C$, and that the quartic $f(v,w,z)$ is given by the equation
\begin{equation}\label{f1}
f(v,w, z)=\epsilon v^2w^2-a_1v^2z^2-a_2w^2z^2+vwz(bv+cw+dz),
\end{equation}
where $\epsilon=\pm 1$, and $a_1, a_2, b,c,d\in R.$
\item If only one node is defined over $R$, up to a linear change of variables, one may assume that the points $[i:1:0]$, $[-i:1:0]$, and $[0:0:1]$ are the nodes of $C$, and that the quartic $f(v,w,z)$ is given by the equation
\begin{equation}\label{f2}
f(v,w, z)=\epsilon(v^2+w^2)^2+(a_1vz+a_2wz)(v^2+w^2)+bv^2z^2+cw^2z^2+dvwz^2,
\end{equation}
where $\epsilon=\pm 1$, and $ a_1, a_2, b,c,d\in R.$
\end{enumerate}
\end{lem}
\begin{proof}
By the genus formula, since the genus of $\Delta$ is zero and the degree $d$ of $f$ is $4$, the curve $C$ has $\frac{(d-1)(d-2)}{2}=3$ nodes. Note that the nodes are not collinear since the degree of $f$ is $4$, by B\'ezout's theorem (as the intersection of $C$ with a line at a node of $C$ is counted with multiplicity $2$). In particular, one node must be defined over $R$, so that we may assume it is the point $[0:0:1]$; the two other nodes are either defined over~$R$, or they are complex conjugate.  

In the former case, we can assume that the two other nodes are $[1:0:0]$ and~$[0:1:0]$. Hence from the partial derivatives condition  the degree of $f$ in each of the variables $v,w,z$ is at most $2$. In addition, since $C$ is irreducible, the coefficient of one of the terms $v^2w^2$, $v^2z^2$, or $w^2z^2$  must be nonzero, so that, up to a permutation of variables and scaling, one may assume that the coefficient of~$v^2w^2$ is $\pm 1$, and we get the equation (\ref{f1}).

In the latter case, we can assume that the two other nodes are $[i:1:0]$ and $[-i:1:0]$.
We obtain the form (\ref{f2}) from the partial derivatives condition again. First, since $[0:0:1]$ is a node of $C$, we can write
$f=f_4+f_3+f_2$, where $$f_4=a_4v^4+b_4w^4+c_4v^3w+d_4vw^3+e_4v^2w^2,$$ $$f_3=a_3v^3z+b_3w^3z+c_3v^2wz+d_3vw^2z,$$ and $f_2=bv^2z^2+cw^2z^2+dvwz^2$.

Then the condition that $[i:1:0]\in C$ reads as $a_4+b_4=e_4, c_4=d_4$.
Conditions that $\partial f/\partial v$ and $\partial f/\partial w$ vanish at $[i:1:0]$ give $c_4=d_4=0$ and $e_4=2a_4=2b_4$.
Condition that $\partial f/\partial z[i:1:0]=0$ gives $a_3=d_3$ and $b_3=c_3$.
Up to a scaling, we may assume $a_4=\pm 1$, hence we get the form (\ref{f2}) as claimed with $a_1\stackrel{def}{=}a_3=d_3$ and $b_1\stackrel{def}{=}b_3=d_3$.
\end{proof}

\begin{thm}
\label{th2}
The algebraic varieties defined by (\ref{eq2}) with $d=4$
 are always rational, over any real closed field.

\end{thm}
\begin{proof}

\begin{enumerate}[label=\textrm{(\alph*)}]
\item We first consider equation \ref{eqb}. By Lemma \ref{elem1}, one can assume that the affine equation of $C$ is 
$$f(w,z)=q(w)z^2+c(w)z+g(w),$$ where the polynomials $q,c,g\in R[w]$ are of degrees at most $2,3,4$ respectively, so that the conic bundle of type~\ref{eqb} is given by the affine equation
$$X:=\{x^2+y^2=f(w,z)\}.$$ 
In particular, $X$ is birational to a quadric surface bundle $\pi:X'\to \mathbb P^1_R$, with generic fiber the irreducible quadric $$Q:=\{x^2+y^2-q(w)z^2-c(w)zt-g(w)t^2=0\}\subset \mathbb P^3_{R(w)}.$$
Since $f(w,z)$ is assumed to be nonnegative, the quadric~$Q$ has a smooth point over~$R(w)$: indeed, taking~$z=z_0\in R$ nonzero and $t=1$, we have that $q(w)z_0^2+c(w)z_0+g(w)\in R[w]$ is nonzero and nonnegative, hence it is a nontrivial sum of two squares in $R[w]$ (it is a product of nonnegative factors of degree $2$; use that each factor is a sum of two squares and that being a sum of two squares is multiplicative). It follows that $X$ is rational.
\item We now consider  equation \ref{eqa}, in the case when all the nodes of $C$ are defined over $R$, so that $C$ is defined by equation (\ref{f1}) from Lemma \ref{elem1}.  Then the conic bundle of type~\ref{eqa} is given by the affine equation
$$X:=\{x^2+y^2=\epsilon v^2w^2-a_1v^2-a_2w^2+vw(bv+cw+d)\}.$$
In particular, $X$ is birational to a quadric surface bundle $X'$ over $\mathbb A^1_v$ given by the equation: 
$$X':=\{x^2+y^2+a_1v^2t^2=\epsilon w^2 (v^2+\epsilon cv-\epsilon a_2)+wt(bv^2+dv)\}.$$

Note that $a_1\neq 0$ and $a_2\neq 0$ since $C$ is irreducible (otherwise the equation of $C$ is divisible by $w$, resp. by $v$).

With the change of variables $w_1=w+\epsilon t\frac{bv^2+dv}{2(v^2+\epsilon cv -\epsilon a_2)}$, $b_1=b/2, d_1=d/2$ we obtain that $X$ is birational to a quadric surface bundle
$$x^2+y^2+a_1v^2t^2+\epsilon v^2t^2\frac{(b_1v+d_1)^2}{v^2+\epsilon cv-\epsilon a_2}=\epsilon (v^2+\epsilon cv-\epsilon a_2)w_1^2.$$
Multiplying by $q(v)=v^2+\epsilon cv-\epsilon a_2$ and setting $w_2=q(v)w_1$, $t_1=vt$ shows that $X$ is birational to a quadric surface bundle
$$
q(v)x^2+q(v)y^2+(a_1q(v)+\epsilon (b_1v+d_1)^2)t_1^2=\epsilon w_2^2.
$$
Multiplying by $\epsilon$ and setting $q_1(v)=-\epsilon q(v)$, $q_2(v)=a_1q_1(v)-(b_1v+d_1)^2$, we see that $X$ is birational to the quadric surface bundle
$$
w_2^2+ q_1(v)x^2 +q_1(v) y^2+q_2(v)t_1^2=0.
$$
We deduce that $X$ is rational by Lemma \ref{elem2} below (note that any $R$-point of $X$ in a fiber over $(v,w)\in R^2$ with $f(v,w)>0$ is smooth).
\item  We now consider  equation \ref{eqa}, in the case when only one  node of $C$ is defined over $R$, so that $C$ is defined by equation (\ref{f2}) from Lemma \ref{elem1}.  Then the conic bundle $X$ of type~\ref{eqa} is given by the affine equation
$$x^2+y^2=\epsilon (v^2+1)^2+z(a_1v+a_2)(v^2+1) + z^2(bv^2+dv+c).$$
Let $q(v)=bv^2+dv+c$ and let $z_1=zq(v)$. Note that $q(v)$ is nonzero since $C$ is irreducible. Multiplying by $q(v)$, we see that $X$ is birational to a quadric surface bundle over $\mathbb A^1_v$ given by the equation
$$q(v)x^2+q(v)y^2=z_1^2+z_1(a_1v+a_2)(v^2+1)t+\epsilon q(v)(v^2+1)^2t^2.$$
Changing  variables to $t_1=(v^2+1)t$, $z_2=z_1+\frac{(a_1v+a_2)(v^2+1)t}{2}$, $X$ is birational to the quadric surface bundle
$$q(v)x^2+q(v)y^2+(-\epsilon q(v)+\frac{(a_1v+a_2)^2}{4})t_1^2=z_2^2.$$
Setting $q_1(v)=-q(v), q_2(v)=-\epsilon q_1(v) -\frac{(a_1v+a_2)^2}{4}$,  and applying Lemma~\ref{elem2} to the above equation (again, any $R$-point of $X$ in a fiber over $(v,w)\in R^2$ with $f(v,w)>0$ is smooth),  we get $X$ is rational in this case as well. \qedhere
\end{enumerate}
\end{proof}

\begin{lem}
\label{elem2}
Let $R$ be a real closed field. Let $q_1(v), q_2(v)\in R[v]$ be polynomials of degree $\leq 2$ with
$q_1(v)$ nonzero and $q_2(v)=aq_1(v)-l_3(v)^2$ for some $a\in R$ and some polynomial $l_3(v)\in R[v]$ of degree $\leq 1$.
Let $X$ be the quadric surface bundle
$$X:= \{w^2+q_1(v)x^2+q_1(v)y^2+q_2(v)t^2=0\}\subset \mathbb P^3_{[w:x:y:t]}\times \mathbb A^1_v.$$
If $X$ has a smooth $R$-point, then $X$ is rational. 
\end{lem}
\begin{proof}
\begin{enumerate}[label=\textrm{(\alph*)}]
\item  Assume first that the sign of $q_1(v)$ does not change.  Then one can write $q_1(v)=\epsilon(l_1(v)^2+l_2(v)^2)$ for some $\epsilon=\pm 1$ and some polynomials $l_1,l_2\in R[v]$ of degree $\leq 1$.  Dehomogenizing with respect to $t$, we obtain that $X$ is birational to a quadric (in the variables $w, x_1, y_1, v$) 
$$w^2+\epsilon x_1^2+\epsilon y_1^2+q_2(v)=0$$ where $x_1=xl_1(v)+yl_2(v)$ and $y_1=xl_2(v)-yl_1(v)$. It follows that $X$ is  rational, as is any irreducible quadric with a smooth rational point.

\item  Assume now that the sign of $q_1(v)$ changes.  Let $q_1(v,s)$ and $l_3(v,s)$ be the homogenizations of $q_1(v)$ and $l_3(v)$. Then the quadratic form $q_1(v,s)$ has two distinct roots in~$\P^1(R)$.  Consequently, after an appropriate change of projective coordinates $[v:s]\mapsto [u:r]$, we may assume that $q_1(u,r)=ur$.  After maybe exchanging $u$ and $r$, we may assume that $l_3(u,r)$ is either zero or depends nontrivially on $u$.
Dehomogenizing with respect to $r$ and $t$ shows that $X$ is birational to 
\begin{equation}\label{aeq}w^2+ux^2+uy^2+(au-l_4(u)^2)=0,
\end{equation}
where $l_4(u)\in R[u]$ is either $0$ or has degree exactly $1$.

If $l_4(u)=0$, expressing $u$ as a rational function of $w,x,y$ shows that $X$ is rational. We may therefore assume that $l_4(u)$ has degree $1$.  In this case, one can write $au-l_4(u)^2=-e^2(u-c)^2+f$ for some $e,c,f\in R$.
Setting $u_1=u-c$, we get that $X$ is birational to the affine variety
\begin{equation*}
w^2+(u_1+c)x^2+(u_1+c)y^2-e^2u_1^2+f=0.
\end{equation*}
We rewrite the last equation as
$$w^2-(eu_1-\frac{x^2}{2e}-\frac{y^2}{2e})^2+h(x,y)=0$$ 
and we make the change of variables $$w_2= w-(eu_1-\frac{x^2}{2e}-\frac{y^2}{2e}), u_2=w+(eu_1-\frac{x^2}{2e}-\frac{y^2}{2e}),$$ so that we obtain that $X$ is birational to
$$w_2u_2+h(x,y)=0$$ which is rational. \qedhere
\end{enumerate}
\end{proof}

\begin{rem}
Note that if $q_2$ is not of the form $aq_1-l_3^2$, the conclusion of Lemma~\ref{elem2} may fail. In particular, a smooth and projective model of the quadric surface bundle defined by the equation
$$w^2-ux^2-uy^2+u^2+1=0$$
 (an analogue of equation (\ref{aeq}) but with different signs)
 has nontrivial Brauer group \cite[Proposition 13.1]{CTP}.
\end{rem}

\section{Preliminaries for the Sarkisov program over nonclosed fields}
\label{sec3}

In this section, we explain how to apply over nonclosed fields some standard tools of birational geometry that will be useful to us in Section \ref{sec4}. 

Let $k$ be a field of characteristic $0$. Let $\ok$ be an algebraic closure of $k$ and let $\Gamma_k:=\Gal(\ok/k)$ be the absolute Galois group of $k$.

\subsection{Running the MMP over nonclosed fields}
\label{parMMP}

We first explain how to run the MMP over possibly nonclosed fields.  We mostly follow \cite{KollarMori}.  We also refer to \cite{PG2} for a detailed discussion in the closely related equivariant setting.

Let $f:X\to Y$ be a projective morphism of varieties over $k$. We refer to \cite[Examples 2.16 and 2.18]{KollarMori} 
for the definitions of the cones of relative $1$\nobreakdash-cycles $\NE(X/Y)\subset\overline{\NE}(X/Y)\subset\NN_1(X/Y)$.  As the $\Gamma_k$\nobreakdash-orbit of a curve in~$X_{\ok}$ is a curve defined over $k$, we see that this sequence of inclusions can be identified with $\NE(X_{\ok}/Y_{\ok})^{\Gamma_k}\subset\overline{\NE}(X_{\ok}/Y_{\ok})^{\Gamma_k}\subset\NN_1(X_{\ok}/Y_{\ok})^{\Gamma_k}$ (note that we do not require curves to be geometrically irreducible).
 
Assume moreover that $(X,\Delta)$ is a log canonical $\Q$-factorial pair over $k$ (see \cite[\S 2.3]{KollarMori} for the definitions of the various classes of singularities that we use). We refer to \cite[\S 3.7]{KollarMori} for what it means to run the MMP for~$(X,\Delta)$ over $Y$. This procedure is still conjectural in general, but it is unconditional in dimension~$3$ as we explain below. It preserves the classes of singularities that we consider (see \cite[Propositions~3.36 and~3.37 and Corollaries 3.42, 3.43 and 3.44]{KollarMori}). It requires four steps: proving a cone theorem,  showing the existence of contractions associated with extremal rays,  showing that the flips of small extremal contractions exist, and showing that no infinite sequences of flips exist. 

Over a nonclosed field $k$, the first step is the following version of the cone theorem.

\begin{thm}
\label{coneth}
There exist countably many rational curves $(C_i)_{i\in I}$ in $X_{\ok}$ whose $\Gamma_k$\nobreakdash-orbits generate $(K_X+\Delta)$\nobreakdash-neg\-a\-tive extremal rays $R_i\subset \overline{\NE}(X/Y)$ and such that
$$\overline{\NE}(X/Y)=\overline{\NE}(X/Y)_{K_X+\Delta\geq 0}+\sum_{i\in I} R_i.$$
In addition, if $H$ is ample on $X$ and $\epsilon>0$, there is a finite subset $J\subset I$ such that
$$\overline{\NE}(X/Y)=\overline{\NE}(X/Y)_{K_X+\Delta+\epsilon H\geq 0}+\sum_{i\in J} R_i.$$
\end{thm}

Over $\ok$, we refer to \cite[Theorem 3.25 (1) and Theorem 3.35]{KollarMori} for the cone theorem in the dlt case and to \cite[Theorem 0.2]{Ambro} or \cite[Theorem 1.1]{Fujino} in the log canonical case. Theorem \ref{coneth} can be deduced formally from the statement over $\ok$ (the argument given in a restricted setting in \cite[Theorem 26]{Kollarreal} works in general). See also \cite[Theorem 3.3.1]{PG2} in the equivariant setting.  We insist that not all rational curves $C\subset X_{\ok}$ generating a $(K_X+\Delta)$\nobreakdash-negative extremal ray of $\overline{\NE}(X_{\ok}/Y_{\ok})$ have the property that their $\Gamma_k$-orbits generate an extremal ray of~$\overline{\NE}(X/Y)$ (see Example \ref{exray}).

\begin{ex}
\label{exray}
Over $k=\R$, take $Y=\Spec(\R)$ and consider the smooth cubic surface $X:=\{X_0(X_1^2+X_2^2+X_3^2)=X_1X_2X_3\}\subset\P^3_{\R}$. The five lines $\ell:=\{X_0=X_1=0\}$, $\ell':=\{X_0=X_2=0\}$, $\ell'':=\{X_0=X_3=0\}$, $\ell^+:=\{X_1=X_2+iX_3=0\}$ and $\ell^-:=\{X_1=X_2-iX_3=0\}$ of $\P^3_{\C}$ are included in $X_{\C}$. As $\ell^+$ is a $(-1)$-curve of~$X_{\C}$, it generates a $K_{X}$-negative extremal ray of $\overline{\NE}(X_{\C}/Y_{\C})$. Its $\Gamma_{\R}$-orbit~$\ell^++\ell^-$ is however not extremal in $\overline{\NE}(X/Y)$, as $\cO_X(\ell^++\ell^-)=\cO_X(1)(-\ell)=\cO_X(\ell'+\ell'')$.
\end{ex}

The second step is the following contraction theorem (applied to extremal rays, that is to one-dimensional extremal faces of $\overline{\NE}(X/Y)$).

\begin{thm}
\label{contractionth}
Let $F$ be a $(K_X+\Delta)$\nobreakdash-neg\-a\-tive extremal face of $\overline{\NE}(X/Y)$. There is a morphism $c_F:X\to X'$ of projective varieties over $Y$ with $\cO_{X'}\isoto (c_F)_*\cO_X$ such that a curve $C\subset X$ contracted in $Y$ is contracted in $X'$ if and only if $[C]\in F$. 
\end{thm}

Again, the statement appears in a restricted setting in \cite[Theorem 27]{Kollarreal}, and in the equivariant setting in \cite[Theorem 3.3.1]{PG2}. To prove it, let $\wF$ be the smallest face of $\overline{\NE}(X_{\ok}/Y_{\ok})$ containing $F$ and note that the contraction of~$\wF$ constructed in \cite[Theorem 3.25 (3), Theorem 3.35]{KollarMori} in the dlt case (and in \cite[Theorem 0.2]{Ambro} or \cite[Theorem 1.1]{Fujino} in the log canonical case) is $\Gamma_k$\nobreakdash-equivariant and hence descends over $k$ to the required morphism $c_F:X\to X'$.

The existence of flips of small contractions follows from the finite generation of appropriate canonical algebras (see \cite[Corollary 6.4 (2)]{KollarMori}),  which holds as a consequence of \cite[Corollary 1.2]{Birkar} or \cite[Corollary 1.8]{HaconXu}.  As for the termination of flips, the full statement is only known when $X$ has dimension $3$. We refer to \cite[Theorem 4.15]{KollarMatsuki} in the canonical case (the only case that we will need); in the log canonical case, one would have to check that \cite[Proof of 5.1.3]{Shokurov3} works over $k$. See also \cite[Theorem 3.4.3, Theorem 4.2.2]{PG2} in the equivariant setting.

There are two reasons why running the MMP over $k$ may be different from running the MMP over $\ok$. The first is that the extremal contractions that appear have relative Picard rank $1$ over $k$, but may have higher relative Picard rank over~$\ok$ (and hence may not be contractions of extremal rays over $\ok$). The second is that all the steps of the MMP over $k$ preserve the condition of being $\Q$-factorial,  but not the condition of being $\Q$\nobreakdash-factorial over $\ok$ in general.  As a consequence, the outcome of the MMP over $k$ may not be  $\Q$\nobreakdash-factorial over $\ok$, in which case it is necessarily distinct from the outcome of the MMP over $\ok$ (see Example \ref{exQfact}).

\begin{ex}
\label{exQfact}
Consider the singular quadric $Y:=\{X_0^2=X_1^2+X_2^2+X_3^2\}\subset\P^4_{\R}$. Let $f:X\to Y$ be the blow-up of the singular point $[0:0:0:0:1]$ of $Y$, so~$X$ is smooth. The exceptional locus $E$ of $f$ is a smooth projective quadric surface with~$E(\R)\simeq\bS^2$. The vector space $\NN_1(X_{\C}/Y_{\C})$ is of dimension $2$, generated by the two rulings of $E_{\C}$. Since these two rulings are exchanged by the action of $\Gamma_{\R}\simeq\Z/2$, the vector space $\NN_1(X_{\C}/Y_{\C})^{\Gamma_{\R}}$ is of dimension $1$. It follows that $f:X\to Y$ is the contraction of an extremal ray, and hence $Y$ is the only possible outcome of running the MMP for $X$ over $Y$ (in particular, the variety $Y$ is $\Q$-factorial).
However, the variety $Y_{\C}$ is not $\Q$-factorial (as the Weil divisor $\{X_0-X_1=X_2-iX_3=0\}$ in~$Y_{\C}$ is not $\Q$-Cartier) and hence is not a possible outcome of running the MMP for~$X_{\C}$ over~$Y_{\C}$. The only two possible outcomes are the two small resolutions of~$Y_{\C}$ obtained from $X_{\C}$ by contracting one of the two rulings of $E_{\C}$.
\end{ex}

\subsection{Mori fiber spaces}
\label{parMFS}

The basic objects of study of the Sarkisov program are Mori fiber spaces and birational maps between them.

\begin{Def}
\label{defMFS}
A morphism $\pi:X\to S$ with geometrically connected fibers between connected normal projective varieties over $k$ is a \textit{Mori fiber space} if
\begin{enumerate}[label=(\roman*)]
\item $X$ has $\Q$-factorial terminal singularities;
\item the fibers of $\pi$ have positive dimension, $\rho(X/S)=1$ and $-K_X$ is $\pi$-ample.
\end{enumerate}
\end{Def}

Some references allow broader classes of singularities on Mori fiber spaces; here we insist that $X$ is terminal.
Let $\phi:X\dashrightarrow X'$ be a birational map between two Mori fiber spaces $\pi:X\to S$  and $\pi':X'\to S'$. The map $\phi$ is 
\textit{square} if it induces 
an isomorphism between the generic fibers of $\pi$ and $\pi'$.  If in addition the map $\phi$ is a (regular) isomorphism, we say that~$\phi$ is an {\it isomorphism} of Mori fiber spaces.

\begin{Def}
\label{defbirrig}
A Mori fiber space $\pi:X\to S$ is \textit{birationally superrigid} if for any Mori fiber space $\pi:X'\to S'$, any birational map $\phi:X\dashrightarrow X'$ is square.
\end{Def}

The Noether--Fano inequalities provide a criterion for $\phi$ to be an isomorphism of Mori fiber spaces. Choose a sufficiently ample divisor $A'$ on $S'$ and a sufficiently large and divisible integer $\mu'$, so that the linear system $\cH':=|-\mu'K_{X'}+(\pi')^{*}A'|$ on $X'$ is very ample. Let $\cH:=(\phi^{-1})_*\cH'$ be the strict transform of $\cH'$ on $X$. Since~$\rho(X/S)=1$ and $-K_X$ is $\pi$-ample, one can find a divisor class $A$ on $S$ and a positive rational number $\mu$ such that $\cH$ is numerically equivalent to $-\mu K_X+\pi^*A$.

\begin{thm}[Noether--Fano inequalities]
\label{NF}
(i) With the above notation, $\mu\geq \mu'$.

(ii) If $(X, \frac{1}{\mu} \mathcal H)$ has canonical singularities and $K_X+\frac{1}{\mu}\mathcal H$ is nef, then $\phi$ induces an isomorphism of Mori fiber spaces. In particular, $\mu=\mu'$.
\end{thm}

The proof of Theorem \ref{NF} given in \cite[Theorem 4.2]{Corti1} when $k$ is algebraically closed works verbatim over possibly nonclosed fields.

\subsection{Sarkisov links}

Sarkisov links $\phi:X\dashrightarrow X'$ are a kind of elementary birational maps between two Mori fiber spaces $\pi:X\to S$ and $\pi':X'\to S'$. 
There are four types of Sarkisov links that are part of a commutative diagram of the form
\begin{equation}
\label{links}
{
\def\arraystretch{1.8}
\begin{array}{cc}
\begin{tikzcd}[column sep=0.8cm,row sep=0.2cm]
Z\ar[dd,"q",swap]  \ar[rr,dotted, "\chi"] && X' \ar[dd,"\pi'"] \\ \\
X\ar[uurr,"\phi",dashed,swap] \ar[dr,"\pi",swap] && S' \ar[dl, "p"] \\
& T=S &
\end{tikzcd}
&
\begin{tikzcd}[column sep=0.4cm,row sep=0.2cm]
Z\ar[dd,"q",swap]  \ar[rr,dotted, "\chi"] && Z' \ar[dd,"q'"] \\ \\
X \ar[rr,"\phi",dashed,swap] \ar[dr,"\pi",swap] && X' \ar[dl,"\pi'"] \\
& T=S=S' &
\end{tikzcd}
\\
\text{Type}~$I$ & \text{Type}~$II$ 
\\
\begin{tikzcd}[column sep=0.8cm,row sep=0.2cm]
X\ar[ddrr,"\phi",dashed,swap] \ar[dd,"\pi",swap]  \ar[rr,dotted, "\chi"] && Z\ar[dd,"q"] \\ \\
S \ar[dr, "p", swap] && X' \ar[dl,"\pi'"] \\
& T=S' &
\end{tikzcd}
&
\begin{tikzcd}[column sep=1.2cm,row sep=0.2cm]
X \ar[rr,"\phi",dotted,swap, "\chi"'] \ar[dd,"\pi",swap]  && X' \ar[dd,"\pi'"] \\ \\
S \ar[dr, "p", swap] && S' \ar[dl, "p'"] \\
& T &
\end{tikzcd}
\\
\text{Type}~$III$ & \text{Type}~$IV$, 
\end{array}
}
\end{equation}
where the morphisms $q$ and $q'$ are divisorial extremal contractions, the morphisms~$p$ and $p'$ are extremal contractions, and the rational maps $\chi$ are a nontrivial composition of log-flips over the base $T$ of the link; note that in dimension $\leq 3$ the contractions $p$ and $p'$ in the link of type IV must be of fiber type. We refer to \cite[Definition 3.8, Remark 3.10]{BLZ}
and \cite[Section 4]{Prokhorov} for more details.

We insist that Mori fiber spaces and Sarkisov links over $k$ may not give rise to Mori fiber spaces and Sarkisov links over $\ok$, because the relative Picard rank one conditions might not be preserved by extension of scalars.

\subsection{Conic bundles}

We will focus on Mori fiber spaces of relative dimension $1$.

\begin{Def}
\label{defQconic}
A $\Q$\textit{-conic bundle} $\pi:X\to S$ is a Mori fiber space whose generic fiber is a conic.
\end{Def}
 
Let $\pi:X\to S$ be a morphism of connected normal projective varieties over~$k$ whose generic fiber is a conic (such as a $\Q$-conic bundle).  The \textit{degeneration di\-vi\-sor}~$\Delta\subset S$ of $\pi$ is the divisorial part of the locus over which $\pi$ is not smooth. 

We introduce the following particular class of $\Q$-conic bundle.

\begin{Def}
\label{defstandard}
A \textit{standard conic bundle} is a flat morphism $\pi:X\to S$ between smooth projective connected varieties over $k$, whose generic fiber is a conic, with strict normal crossings (snc) degeneration divisor $\Delta\subset S$, and such that $\rho(X/S)=1$. 
\end{Def}

Unlike other references, we choose to require that $\Delta$ has simple normal crossing singularities.
The next theorem provides standard models for arbitrary conic bundles.  Over~$\ok$, it is due to Sarkisov \cite[Theorem 1.13]{Sarkisovconic}, 
and over $k$ to Avilov \cite[Theorem~1]{Avilov} in dimension $3$ and to Oesinghaus \cite[Theorem 4]{Oesinghaus} in general. 

\begin{thm}
\label{thstandard}
Let $\pi:X\to S$ be a morphism of connected normal projective varieties over $k$ whose generic fiber is a conic.  Then there exist a standard conic bundle $\pi^{\bullet}:X^\bullet\to S^{\bullet}$,  a birational map $\xi:X^\bullet\dashrightarrow X$ and a birational morphism $\zeta:S^\bullet\to S$  such that 
$\zeta\circ \pi^\bullet=\pi\circ \xi$.
\end{thm}

The next lemma is well-known.

\begin{lem}
\label{lemconic1}
Let $\pi:X\to S$ be a morphism of connected normal projective varieties over $k$ whose generic fiber is a conic with class $\alpha\in\Br(k(S))$. Let~$\Delta\subset S$ be the degeneration divisor of $\pi$. Let $R\subset S$ be the ramification divisor of $\alpha$.

 One has $R\subset\Delta$.
If moreover $\pi$ is a standard conic bundle, then $R=\Delta$.
\end{lem}

\begin{proof}
Let $s\in S$ be the generic point of a divisor.  Let $\pi_s:X_s\to\Spec(\cO_{S,s})$ be the restriction of $\pi$. 
If $s\notin \Delta$, then $\pi_s$ is a smooth conic, so $\alpha$ lifts to $\Br(\cO_{S,s})$ and~$s\notin R$. Hence $R\subset \Delta$.
If $\pi$ is standard and~$s\in \Delta$, then $\pi_s$ is a relatively minimal regular surface which is not smooth.  By uniqueness of relatively minimal models, the generic fiber of~$\pi$ cannot extend to a smooth conic over $\Spec(\cO_{S,s})$. This shows that $s\in R$.
\end{proof}

The following lemma is \cite[Proposition 4.2]{Sarkisov}. We include a proof to emphasize that it holds over $k$.

\begin{lem}
\label{lemconic2}
Let $\pi:X\to S$ and $\pi^{\bullet}:X^\bullet\to S^{\bullet}$ be two standard conic bundles of dimension $3$ over $k$ with degeneration divisors $\Delta$ and $\Delta^{\bullet}$.  Let $\xi:X^\bullet\dashrightarrow X$ be a birational map and $\zeta:S^\bullet\to S$ be a birational morphism such that $\zeta\circ \pi^\bullet=\pi\circ \xi$.
Fix $m\geq 2$. If $|mK_{S}+\Delta|\neq\varnothing$ then $|mK_{S^{\bullet}}+\Delta^{\bullet}|\neq\varnothing$.
\end{lem}

\begin{proof}
Let $\alpha\in\Br(k(S))=\Br(k(S^{\bullet}))$ be the class of the generic fibers of $\pi$ and~$\pi^{\bullet}$.
Let $R\subset S$ and $R^{\bullet}\subset S^{\bullet}$ be its ramification divisors.
By Lemma \ref{lemconic1},  one has~$R=\Delta$ and $R^{\bullet}=\Delta^{\bullet}$. These divisors are therefore snc. Under these hypotheses, our goal is to show that if $|mK_{S}+R|\neq\varnothing$ then $|mK_{S^{\bullet}}+R^{\bullet}|\neq\varnothing$.

Writing $\zeta:S^\bullet\to S$ as a composition of blow-ups of closed points, we reduce to the case where $\zeta$ is the blow-up of a single closed point with exceptional divisor $E\subset S^{\bullet}$. Then $K_{S^{\bullet}}=\zeta^*K_S+E$ and $R^{\bullet}=\zeta^*R+aE$ with $a\in\{-2,-1,0\}$ (depending on how many components of $R$ contain the blow-up center and on whether $E\subset R^{\bullet}$ or not). Consequently, $$mK_{S^{\bullet}}+R^{\bullet}=\zeta^*(mK_{S}+R)+(a+m)E.$$
As~$a+m\geq 0$, the hypothesis that $|mK_{S}+R|\neq\varnothing$ implies $|mK_{S^{\bullet}}+R^{\bullet}|\neq\varnothing$.
\end{proof}

\begin{rem}
A more precise analysis would show that Lemma \ref{lemconic2} remains true for $m=1$ (over $k=\C$, see \cite[Lemma 1]{Iskovskikhconic}). We will not need this refinement.
\end{rem}

\section{Birational rigidity over nonclosed fields}
\label{sec4}

We now explain how to adapt the proof of Sarkisov's theorem over nonclosed fields (see Theorem \ref{th4}).
As in Section \ref{sec3}, we let $k$ be a field of characteristic $0$ with algebraic closure $\ok$ and $\Gamma_k:=\Gal(\ok/k)$ be its absolute Galois group.

\subsection{Running the Sarkisov program}

In this paragraph,  following Corti \cite{Corti1}, we explain
 how to run the Sarkisov program in dimension $3$.
Although some version of the Sarkisov program is now known in arbitrary dimension \cite[Theorem~1.1]{HX}, it is the precise algorithm described in \cite{Corti1} and sketched in Lemmas~\ref{Sarkisov1} and \ref{Sarkisov2} below that is used in applications to birational rigidity.

\begin{thm}
\label{thSarkisovprogram}
A birational map $\phi:X\dashrightarrow X'$ between Mori fiber spaces $\pi:X\to S$ and $\pi':X'\to S'$ of dimension $3$ over $k$ is a composition of Sarkisov links.
\end{thm}

\begin{proof}
As in \S\ref{parMFS},  we choose a sufficiently ample divisor $A'$ on $S'$ and a sufficiently large and divisible integer $\mu'$, so that the linear system $\cH':=|-\mu'K_{X'}+(\pi')^{*}A'|$ on~$X'$ is very ample. 
We let $A$ be a divisor class on $S$ and $\mu$ be a positive rational number such that $\cH:=(\phi^{-1})_*\cH'$ is numerically equivalent to $-\mu K_X+\pi^*A$ (and we use similar notation for other Mori fiber spaces birational to $X'$).

Lemmas \ref{Sarkisov1} and \ref{Sarkisov2} provide a composition of Sarkisov links $X\dashrightarrow \check{X}$ between $\pi:X\to S$ and another Mori fiber space $\check{\pi}:\check{X}\to\check{S}$ such that either $\check{\mu}<\mu$ 
(where~$\check{\mu}$ is defined in the same way as above)
or the induced birational map $\check{X}\dashrightarrow X'$ is an isomorphism of Mori fiber spaces. 
As the rational numbers $\mu$ that can appear have bounded denominators by \cite{Kawamata}, iterating this process finitely many times proves the theorem.
\end{proof}

\begin{lem}
\label{Sarkisov1}
There exist a Mori fiber space $\hat{\pi}:\whX\to\whS$ and a composition ${\hat{\phi}:X\dashrightarrow \whX}$ of Sarkisov links of types I and II such that $\hat{\mu}\leq\mu$ and $(\whX,\frac{1}{\hat{\mu}}\hat{\cH})$ has canonical singularities.
\end{lem}

\begin{proof}
If $(X,\frac{1}{\mu}\cH)$ does not have canonical singularities, then  \cite[(5.4.1)]{Corti1} describes a procedure to construct a Sarkisov link $X\dashrightarrow X_1$ of type I or II between $\pi:X\to S$ and another Mori fiber space $\pi_1:X_1\to S_1$.  This construction works over a nonclosed field $k$ as one can apply the MMP over $k$ (see \S\ref{parMMP}). In particular, the extremal blowup that is used in the proof exists over $k$ since its construction in \cite[Proposition~2.10]{Corti1} only depends on running the MMP.

The procedure decreases a triple of invariants $(\mu,c,e)$ that is introduced there.  As this triple cannot decrease infinitely many times (this is shown in \cite[Theorem~6.1]{Corti1} but it nowadays suffices to apply the ACC property for canonical thresholds in dimension $3$ proven by Chen \cite[Theorem 1.2]{Chen}), iterating the procedure finitely many times yields the desired Mori fiber space $\hat{\pi}:\hat{X}\to\whS$.
\end{proof}

\begin{lem}
\label{Sarkisov2}
Let $\hat{\pi}:\whX\to\whS$ be as constructed in Lemma \ref{Sarkisov1}.
There exist a Mori fiber space $\check{\pi}:\check{X}\to\check{S}$ and a composition ${\check{\phi}:\whX\dashrightarrow \check{X}}$ of Sarkisov links of types III and IV such that either $\check{\mu}<\hat{\mu}$, or $\check{\phi}$ is square and $\phi\circ\hat{\phi}^{-1}\circ\check{\phi}^{-1}:\check{X}\dashrightarrow X'$ is an isomorphism of Mori fiber spaces. 
\end{lem}

\begin{proof}
If $\phi\circ\hat{\phi}^{-1}$ is not an isomorphism of Mori fiber spaces, then \cite[(5.4.2)]{Corti1} describes a procedure to construct a Sarkisov link $\psi:\whX\dashrightarrow\whX_1$ of type III or~IV between $\hat{\pi}:\whX\to\whS$ and another Mori fiber space $\hat{\pi}_1:\whX_1\to\whS_1$ such that either~$\hat{\mu}_1<\hat{\mu}$ and $\psi$ is square, or~$\hat{\mu}_1=\hat{\mu}$,  the pair $(\whX_1,\frac{1}{\hat{\mu}_1})$ is still canonical and $\psi$ is not square. This construction works over~$k$ by applying the MMP as explained in~\S\ref{parMMP} (and most crucially the cone and contraction theorems as stated in Theorems~\ref{coneth} and~\ref{contractionth}) and by making use of the Noether--Fano inequalities in the form of Theorem~\ref{NF}.  One then iterates this procedure until either the invariant~$\mu$ drops,  or one reaches the Mori fiber space ${\pi':X'\to S'}$ (in which case $\check{\phi}$ is square).  As is explained at the end of \cite[(5.4.2)]{Corti1}, this process must indeed terminate after finitely many steps by termination of canonical flips in dimension~$3$.
\end{proof}

\subsection{Birational rigidity of conic bundles}

Here is at last Sarkisov's birational rigidity theorem \cite{Sarkisov}. We follow the exposition of Prokhorov \cite[Section~12]{Prokhorov}.

\begin{thm}
\label{th4}
Let $k$ be a field of characteristic $0$. Let $\pi:X\to S$ be a standard conic bundle of dimension $3$ over $k$ with discriminant curve $\Delta\subset S$. 
Assume that $|4K_S+\Delta|\neq\varnothing$. Then $\pi:X\to S$ is birationally superrigid.
\end{thm}

\begin{proof}
Let $\pi':X'\to S'$ be a Mori fiber space and let $\phi:X\dashrightarrow X'$ be a birational map. Assume for contradiction that $\phi$ is not square.

Write $\phi$ as a composition of Sarkisov links using the algorithm of the proof of Theorem \ref{thSarkisovprogram}. We keep the notation introduced there and we let $\hat{\pi}:\whX\to\whS$ and $\check{\pi}:\check{X}\to\check{S}$ be as in Lemmas \ref{Sarkisov1} and \ref{Sarkisov2}.  A Sarkisov link of type I or II between a $\Q$-conic bundle and another Mori fiber space must be square and lead to another $\Q$-conic bundle (in view of the diagrams (\ref{links})). It follows that $\hat{\pi}:\whX\to\whS$ is a $\Q$-conic bundle and that the induced birational map $\whX\dashrightarrow X'$ is not square.

Consequently,  the links appearing in the factorization of ${\check{\phi}}$ as a composition of Sarkisov links provided by Lemma \ref{Sarkisov2} are all square,  except for the last one which decreases the invariant $\mu$. We denote this last link (of type~III or IV) by $\psi:\bar{X}\dashrightarrow\check{X}$ (between some $\Q$-conic bundles $\bar{\pi}:\bar{X}\to\bar{S}$ and $\check{\pi}:\check{X}\to\check{S}$). 
As $\psi$ has type~III or~IV and is not square, its base $T$ (see (\ref{links})) must be a point or a curve. 

Let $\bar{\Delta}\subset\bar{S}$ be the degeneration divisor of $\bar{\pi}$.
We claim that $4K_{\bar{S}}+\bar{\Delta}$ is not effective.
 Indeed,  it follows from \cite[Lemma 3.11]{Prokhorov} (applied over $\ok$) that $-\bar{\pi}_*K_{\bar{X}}^2\equiv 4K_{\bar{S}}+\bar{\Delta}$.  The projection formula then implies that
\begin{equation}
\label{pushforward}
\bar{\pi}_*(\bar{\cH}^2)\equiv 4\bar{\mu}\bar{A}-\bar{\mu}^2(4K_{\bar{S}}+\bar{\Delta}).
\end{equation} 

If the base of $\psi$ is a point, then $\rho(\bar{S})=1$. Since $\bar{A}$ is not nef by Theorem~\ref{NF}~(ii) (hence antiample as $\rho(\bar{S})=1$) and $\bar{\pi}_*(\bar{\cH}^2)$ is effective (hence nef as $\rho(\bar{S})=1$), we see from (\ref{pushforward}) that $-(4K_{\bar{S}}+\bar{\Delta})$ is ample. In particular, 
$4K_{\bar{S}}+\bar{\Delta}$ is not effective.

If the base $T$ of $\psi$ is a curve,  we let $F$ denote the class of a fiber of the extremal contraction $p:\bar{S}\to T$.  The very construction of the link $\psi$ given in \cite[(5.4.2)]{Corti1} implies that $\bar{A}\cdot F<0$. In addition,  $\bar{\pi}_*(\bar{\cH}^2)\cdot F\geq 0$ because $\bar{\pi}_*(\bar{\cH}^2)$ is effective and the fibers of $p:\bar{S}\to T$ cover $\bar{S}$.  We deduce from (\ref{pushforward}) that $(4K_{\bar{S}}+\bar{\Delta})\cdot F<0$. Consequently, the class $4K_{\bar{S}}+\bar{\Delta}$ is not effective in this case either.

To conclude, we use Theorem \ref{thstandard} to find a standard conic bundle $\pi^{\bullet}:X^\bullet\to S^{\bullet}$, as well as birational maps $\xi:X^\bullet\dashrightarrow X$ and $\bar{\xi}:X^\bullet\dashrightarrow \bar{X}$
and birational morphisms $\zeta:S^\bullet\to S$ and $\bar{\zeta}:S^\bullet\to \bar{S}$ such that 
$\zeta\circ \pi^\bullet=\pi\circ\xi$ and $\bar{\zeta}\circ \pi^\bullet=\bar{\pi}\circ\bar{\xi}$.

 Let $\Delta$ and~$\Delta^{\bullet}$ be the degeneration divisors of $\pi$ and $\pi^{\bullet}$.  As ${|4K_S+\Delta|\neq\varnothing}$, Lemma~\ref{lemconic2} applied with $m=4$ implies that $|4K_{S^{\bullet}}+\Delta^{\bullet}|\neq\varnothing$. Since $\bar{S}$ is normal, there is an open subset ${U\subset \bar{S}}$ whose complement has codimension $\geq 2$ such that~$\bar{\zeta}|_{\bar{\zeta}^{-1}(U)}:\bar{\zeta}^{-1}(U)\to U$ is an isomorphism. As~$\bar{\zeta}_*\Delta^{\bullet}\subset\bar{\Delta}$ (because $\bar{\zeta}_*\Delta^{\bullet}$ is the ramification divisor of the Brauer class of the generic fiber of $\bar{\pi}$ by Lemma \ref{lemconic1}), any nonzero section of~$4K_{S^{\bullet}}+\Delta^{\bullet}$ can be restricted to $\zeta^{-1}(U)$, then descended to~$U$, and finally extended to~$\bar{S}$ to give rise to a nonzero section of~$4K_{\bar{S}}+\bar{\Delta}$. This is the required contradiction.
\end{proof}

\subsection{An application}

In this paragraph, we apply Theorem \ref{th4} to prove the following result, which recovers Theorem \ref{th3intro} when $k$ is a real closed field and $a=-1$.

\begin{thm}
\label{th3}
Let $k$ be a field of characteristic $0$.  Fix $a\in k^*\setminus (k^*)^2$ and $f\in k[v,w]$ of even degree $d\geq 12$. If the closure $C\subset\P^2_k$ of
 the affine curve $\{f=0\}$ is nodal, then the algebraic variety with affine equation
$x^2-ay^2=f(v,w)$ is $k[\sqrt{a}]$-rational but not $k$-rational.
\end{thm}

\begin{proof}
The conic $x^2-ay^2=f$ has a rational point given by $x=\sqrt{a}$ and $y=1$ over~$k[\sqrt{a}](v,w)$ and hence is rational over this field.  This proves the first assertion.

Let $\nu:S\to\P^2_k$ be the blow-up of $\P^2_k$ along the nodes of $C$,  let~$E\subset S$ be its exceptional divisor, and let $\Delta\subset S$ be the strict transform of $C$.  
As $K_S=\nu^*K_{\P^2_k}+E$ and $\Delta=\nu^*C-2E$,  one computes that the divisor 
$$4K_S+\Delta=\nu^*(4K_{\P^2_k}+C)+2E=\nu^*\cO_{\P^2_k}(d-12)+2E$$
is effective because $d\geq 12$. The second assertion therefore follows from Theorem~\ref{th4} applied to the
birational model $\pi:X\to S$ (with $S$ defined as above) 
provided by Lemma \ref{lemstandardmodel} below.  
\end{proof}

\begin{lem}
\label{lemstandardmodel}
The variety considered in Theorem \ref{th3} is birational to a standard conic bundle $\pi:X\to S$ with degeneration divisor $\Delta\subset S$.
\end{lem}

\begin{proof}
Consider the line bundle $\cL:=\nu^*\cO_{\P^2_k}(\frac{d}{2})(-E)$ on $S$, where  $\nu:S\to\P^2_k$ is the blow-up of $\P^2_k$ along the nodes of $C$.
As $\Delta=\nu^*C-2E$, there is a section $\sigma\in H^0(S,\cL^{\otimes 2})$ such that $\Delta=\{\sigma=0\}$. Set $\cE:=\cL^{-1}\oplus\cL^{-1}\oplus\cO_S$ and let $p:\P(\cE)\to S$ be its projectivization in the sense of Grothendieck. 

Then $p_*(\cO_{\P(\cE)}(1))=\cL^{-1}\oplus\cL^{-1}\oplus\cO_S$. The three summands of this decomposition give rise to sections $x,y\in H^0(\P(\cE),p^*\cL\otimes\cO_{\P(\cE)}(1))$ and $z\in H^0(\P(\cE),\cO_{\P(\cE)}(1))$.
Let $X\subset\P(\cE)$ be defined by the equation $\{x^2-ay^2=\sigma z^2\}$, and let $\pi:X\to S$ be the restriction of $p$.  It is clear that $X$ is birational to the variety we consider.

To check that $\pi$ is standard with degeneration divisor $\Delta$, the only nontrivial condition one has to verify is that $\rho(X/S)=1$.  This follows from the fact that the inverse image in $X$ of any irreducible divisor $D\subset S$ is irreducible.  If $D\not\subset\Delta$, this is obvious.  If $D\subset\Delta$, this results from the hypothesis that $a\notin (k^*)^2$.
\end{proof}

\bibliographystyle{myamsalpha}
\bibliography{rationality}

\end{document}